\documentclass{amsart}
\usepackage{cases}
\usepackage{amsfonts}
\usepackage{graphicx}
\usepackage{stmaryrd}
\usepackage{amssymb}
\usepackage{mathrsfs}
\usepackage{amsthm}
\usepackage{amsmath}
\usepackage{bbm}

\newtheorem{thm}{Theorem}

\newtheorem{lem}[thm]{Lemma}
\newtheorem*{keylem}{Key Lemma}
\newtheorem*{thma}{Theorem A}
\newtheorem*{thmb}{Theorem B}

\newcommand{\hyperg}[4]{\: _2\! F_1\! \left[ \begin{array}{c} #1,\, #2 \\ #3 \end{array} ;\,  #4 \right]}
\newcommand{\RePt}{\mathrm{Re}\,}
\newcommand{\ImPt}{\mathrm{Im}\,}
\newcommand{\ball}{\mathbb{B}}
\newcommand{\sphere}{\mathbb{S}}

\newcommand{\calU}{\mathcal{U}}

\newcommand{\bfi}{\mathbf{i}}

\newcommand{\bfrho}{\boldsymbol{\rho}}

\newcommand{\bbB}{\mathbb{B}}
\newcommand{\bbC}{\mathbb{C}}
\newcommand{\bbR}{\mathbb{R}}
\newcommand{\bbH}{\mathbb{H}}
\newcommand{\bbN}{\mathbb{N}}

\begin{document}

\title[Two classes of integral operators]{Two classes of integral operators over the Siegel upper half-space}

\thanks{The first author was supported by the National Natural Science
Foundation of China grants 11571333, 11471301; the fourth author was supported by
Natural Science Foundation of Zhejiang province grant (No. LQ13A010005), the Scientific
Research and Teachers project of Huzhou University(No. RP21028) and partially by the
National Natural Science Foundation of China grant 11571105.}

\author[C. Liu]{Congwen Liu}

\address{School of Mathematical Sciences,
         University of Science and Technology of China,
         Hefei, Anhui 230026,
         People's Republic of China.\\
and \\
Wu Wen-Tsun Key Laboratory of Mathematics,
USTC, Chinese Academy of Sciences, Hefei,
         People's Republic of China.}

\email{cwliu@ustc.edu.cn}

\author[Y. Liu]{Yi Liu}
\address{School of Mathematical Sciences,
         University of Science and Technology of China,
         Hefei, Anhui 230026,
         People's Republic of China.}
\email{ly0717@mail.ustc.edu.cn}

\author[P. Hu]{Pengyan Hu}
\address{Colledge of Mathematics and Statistics,
         Shenzhen University, Shenzhen, Guangdong 518060,
         People's Republic of China.}

\email{pyhu@szu.edu.cn}

\author[L. Zhou]{Lifang Zhou}
\address{Department of Mathematics,
Huzhou University,
Huzhou, Zhejiang 313000,
People¡¯s Republic of China}
\email{lfzhou@zjhu.edu.cn}

\begin{abstract}
We determine exactly when two classes of integral operators are bounded on weighted $L^p$ spaces
over the Siegel upper half-space.
\end{abstract}

\maketitle

\section{Introduction}
\label{intro}

This short note is motivated by the work of Kures and Zhu \cite{KZ06}, in which the authors characterized
the boundedness of two classes of integral operators induced by Bergman type kernels on weighted
Lebesgue spaces on the unit ball $\ball$ of $\mathbb{C}^n$.

Fix three real parameters $a, b, c$ and define two integral operators $\mathcal{T}_{a,b,c}$
and $\mathcal{S}_{a,b,c}$ by
\begin{align*}
\mathcal{T}_{a,b,c} f(z) ~:=~& (1-|z|^2)^{a} \int\limits_{\ball} \frac { (1-|w|^2)^{b}} { (1-\langle z,w\rangle)^{c}} f(w) d\nu(w)\\
\intertext{and}
\mathcal{S}_{a,b,c} f(z) ~:=~& (1-|z|^2)^{a} \int\limits_{\ball} \frac { (1-|w|^2)^{b}} { |1-\langle z,w\rangle|^{c}} f(w) d\nu(w),
\end{align*}
where $d\nu$ is the volume measure on $\ball$, normalized so that $\nu(\ball)=1$.
Also, for any real parameter $\alpha$ we define $d\nu_{\alpha}(z):= (1-|z|^2)^{\alpha} d\nu(z)$.

Kures and Zhu \cite{KZ06} obtained the following two theorems.

\begin{thma}
Suppose $1< p < \infty$. Then the following conditions are equivalent:
\begin{enumerate}
\item[(i)]
The operator $\mathcal{T}_{a,b,c}$ is bounded on $L^p(\ball, d\nu_{\alpha})$.
\item[(ii)]
The operator $\mathcal{S}_{a,b,c}$ is bounded on $L^p(\ball, d\nu_{\alpha})$.
\item[(iii)]
The parameters satisfy
\[
\begin{cases} -pa<\alpha+1<p(b+1) \\
c\leq n+1+a+b.
\end{cases}
\]
\end{enumerate}
\end{thma}

\begin{thmb}
The following conditions are equivalent:
\begin{enumerate}
\item[(i)]
The operator $\mathcal{T}_{a,b,c}$ is bounded on $L^1(\ball, d\nu_{\alpha})$.
\item[(ii)]
The operator $\mathcal{S}_{a,b,c}$ is bounded on $L^1(\ball, d\nu_{\alpha})$.
\item[(iii)]
The parameters satisfy
\[
\begin{cases} -a<\alpha+1< b+1  \\
c = n+1+a+b.
\end{cases}
\quad \text{ or } \quad
\begin{cases} -a<\alpha+1 \leq b+1 \\
c < n+1+a+b.
\end{cases}
\]
\end{enumerate}
\end{thmb}

Actually, these two theorems were proved in \cite{KZ06} under the additional assumption that $c$ is neither $0$
nor a negative integer. Recently, Zhao \cite{Zha15} removed this extra requirement as well as generalized these two
theorems by characterizing the boundedness of $\mathcal{T}_{a,b,c}$ and $\mathcal{S}_{a,b,c}$, from
$L^p(\ball, d\nu_{\alpha})$ to $L^q(\ball, d\nu_{\beta})$.

The case $c=n+1+a+b$ of Theorems A is well known and being extensively used, see for example \cite[Theorem 2.10]{Zhu05}.
It is also worthy to mention that, recently, a variant of Theorem A played a crucial role in the proof of the corona theorem for the Drury-Arveson
Hardy space, see \cite[Lemma 24]{CSW11}.

In this note we consider the counterparts of Theorems A and B for two classes of integral operators over
the Siegel upper half-space. The situation turns out to be quite different in this setting.

Before stating our main result, we introduce some definitions and notation.

We fix a positive integer $n$ throughout this paper and let $\bbC^n = \bbC\times \cdots \times \bbC$
denote the $n$-dimensional complex Euclidean space. For any two points $z=(z_1,\cdots,z_n)$
and $w=(w_1,\cdots,w_n)$ in $\bbC^n$, we write
\[
\langle z,w\rangle := z_1\bar{w}_1 + \cdots + z_n\bar{w}_n
\]
and $|z|:=\sqrt{\langle z,z\rangle}$. The open unit ball in $\bbC^n$ is the set
\[
\ball := \{ z\in \bbC^n: |z|<1\}.
\]
For $z\in \bbC^n$, we also use the notation
\[
z=(z^{\prime},z_{n}), \quad \text{where } z^{\prime}=(z_1,\ldots,z_{n-1})\in \bbC^{n-1} \text{ and } z_{n}\in \bbC^1.
\]
The Siegel upper half-space in $\bbC^n$ is the set
\[
\calU := \left\{ z\in \bbC^n: \ImPt z_{n} > |z^{\prime}|^2 \right\}.
\]
It is biholomorphically equivalent to the unit ball $\ball$ in
$\bbC^n$, via the Cayley transform $\Phi:\bbB \to \calU$ given by
\[
(z^{\prime}, z_{n})\; \longmapsto\; \left( \frac {z^{\prime}}{1+z_{n}},
i\frac {1-z_{n}}{1+z_{n}} \right),
\]
and so it is also referred to as the unbounded realization of the unit ball in $\bbC^n$.

We denote by $dV$ the Lebesgue measure on $\bbC^{n}$. For any real parameters $a$, $b$, and $c$, we consider two integral operators as follows.
\[
T_{a,b,c} f(z) := \bfrho(z)^{a} \int\limits_{\calU} \frac { \bfrho (w)^{b}} { \bfrho(z,w)^{c}} f(w) dV(w)
\]
and
\[
S_{a,b,c} f(z) := \bfrho(z)^{a} \int\limits_{\calU} \frac { \bfrho (w)^{b}} {|\bfrho(z,w)|^{c}} f(w) dV(w),
\]
where
\begin{equation*}
\bfrho(z,w) ~:=~ \frac {i}{2} (\bar{w}_{n}-z_{n})
- \langle z^{\prime}, w^{\prime} \rangle.
\end{equation*}
and $\bfrho(z):=\bfrho(z,z)= \ImPt z_n - |z^{\prime}|^2$.
These operators are modelled on the weighted Bergman projections on $\calU$.
Recall that the Bergman projection $P$ on $\calU$ is given by
\begin{equation*}
P f(z) =  \frac {n!}{4\pi^{n}}\, \int\limits_{\calU} \frac {f(w)} {\bfrho(z,w)^{n+1}} dV(w)
=  \frac {n!}{4\pi^{n}}\, T_{0,0,n+1} f(z), \quad z\in \calU.
\end{equation*}
See, for instance, \cite[Proposition 5.1]{Gin64}.
%

For real parameter $\alpha$, we define
\[
dV_{\alpha}(z):=\bfrho(z)^{\alpha} dV(z).
\]
As usual, for $p>0$, the space $L^p(\calU, dV_{\alpha})$ consists of all Lebesgue measurable
functions $f$ on $\calU$ for which
\begin{equation*}
\|f\|_{p,\alpha}:=\bigg\{\int\limits_{\calU} |f(z)|^p dV_{\alpha}(z)\bigg\}^{1/p}
\end{equation*}
is finite.

Our main result gives necessary and sufficient conditions for the
boundedness of the operators $S_{a,b,c}$ and $T_{a,b,c}$ on $L^p(\calU, dV_{\alpha})$ in terms of
parameters $a, b, c$, and $\alpha$.


\begin{thm}\label{thm:main}
Suppose $\alpha\in \bbR$ and $1\leq p \leq \infty$. Then the following conditions are equivalent:
\begin{enumerate}
\item[(i)]
The operator $T=T_{a,b,c}$ is bounded on $L^p(\calU, dV_{\alpha})$.
\item[(ii)]
The operator $S=S_{a,b,c}$ is bounded on $L^p(\calU, dV_{\alpha})$.
\item[(iii)]
The parameters satisfy the conditions
\begin{equation}\label{eqn:condns}
\begin{cases}
-pa<\alpha+1<p(b+1), \\
c=n+1+a+b.
\end{cases}
\end{equation}
When $p=\infty$, these conditions should be interpreted as
\begin{equation}\label{eqn:condns4infty}
\begin{cases}
a>0, \quad b>-1,\\
c=n+1+a+b.
\end{cases}
\end{equation}
\end{enumerate}
\end{thm}

Note that Condition (iii) in Theorem \ref{thm:main} is different
from the corresponding ones in Theorems A and B. In particular, unlike $\mathcal{T}_{a,b,c}$
and $\mathcal{S}_{a,b,c}$, both $T_{a,b,c}$ and $S_{a,b,c}$ are unbounded whenever $c\neq n+1+a+b$.
This is due to the unboundedness of the Siegel upper half-space and the homogeneity of
the operators $T_{a,b,c}$ and $S_{a,b,c}$.

The proof follows the same main lines as in \cite{KZ06}. However, the computations here are more subtle.
For instance, in the proof of the necessity for the boundedness of $T_{a,b,c}$, we cannot simply choose
polynomials to serve as test functions as in \cite{KZ06}, since polynomials do not belong to
$L^p(\calU, dV_{\alpha})$. Instead, we consider the functions of the form $\bfrho(z)^t/\bfrho(z,w)^s$,
with appropriate choices of the parameters involved.  This leads to more complicated calculations than those
arising in the unit ball setting. Hence, an essential role is played by the following lemma,
which might be of independent interest.

\begin{keylem}\label{lem:keylem2}
Suppose that $r,\,s>0$, $t>-1$ and $r+s-t>n+1$. Then
\begin{equation}\label{eqn:keylem2}
\int\limits_{\calU}  \frac {\bfrho(w)^{t}} {\bfrho(z,w)^{r} \bfrho(w,u)^{s}} dV(w)
~=~ \frac {C_1(n,r,s,t)} {\bfrho(z,u)^{r+s-t-n-1}}
\end{equation}
holds for all $z, u\in \calU$, where
\begin{equation}\label{eqn:const}
C_1(n,r,s,t) ~:=~  \frac {4\pi^{n} \Gamma(1+t)\Gamma(r+s-t-n-1)}{\Gamma(r)
\Gamma(s)}.
\end{equation}
\end{keylem}

The formula \eqref{eqn:keylem2}, with implicit constant $C_1(n,r,s,t)$, is not new;  it is a special case of \cite[Lemma 2.2']{CR80}.
The novelty here is to find the explicit expression \eqref{eqn:const} of $C_1(n,r,s,t)$.

The rest of the paper is organized as follows: In Section 2 we recall some basic materials about M\"obius transformations
and the Cayley transform. Section 3 is devoted to the proof of Key Lemma. Our main result, Theorem \ref{thm:main}
will be proved in Sections 4. Finally, in Section 5, two examples are given to illustrate the use of Theorem \ref{thm:main}.

\section{Preliminaries}

We begin by recalling that the Cayley transform $\Phi:\bbB \to \calU$ is given by
\[
(z^{\prime}, z_{n})\; \longmapsto\; \left( \frac {z^{\prime}}{1+z_{n}},
i\left(\frac {1-z_{n}}{1+z_{n}}\right) \right).
\]
It is easy to check that the identity
\begin{equation}\label{eqn:identity14phi}
\bfrho(\Phi(\eta),\Phi(\xi)) = \frac {1-\langle \eta, \xi\rangle} {(1+\eta_{n}) (1+\overline{\xi}_{n})}
\end{equation}
holds for all $\eta,\xi\in \ball$, and the real Jacobian of $\Phi$ at $\xi\in \ball$ is
\begin{equation}\label{eqn:jacobian4phi}
\left(J_{R}\Phi\right)(\xi) = \frac {4}{|1+\xi_{n}|^{2(n+1)}}.
\end{equation}

The group of all one-to-one holomorphic mappings of $\ball$ onto $\ball$ (the so-called automorphisms
of $\ball$) will be denoted by $\mathrm{Aut}(\ball)$. It is generated by the unitary transformations
on $\bbC^{n}$ along with the M\"obius transformations
$\varphi_{\eta}$ given by
\[
\varphi_{\eta}(\xi) := \frac {\eta-P_{\eta}\xi-(1-|\eta|^2)^{\frac {1}{2}}Q_{\eta}\xi}
{1-\langle \xi, \eta\rangle},
\]
where $\eta\in \ball$, $P_{\eta}$ is the orthogonal projection onto the space spanned
by $\eta$,
and $Q_{\eta}\xi=\xi-P_{\eta}\xi$.

It is easily shown that the mapping $\varphi_{\eta}$ satisfies
\[
\varphi_{\eta}(0)=\eta, \quad \varphi_{\eta}(\eta)=0, \quad \varphi_{\eta}(\varphi_{\eta}(\xi))=\xi.
\]
Furthermore, for all $\xi, \zeta\in \ball$,
\begin{align}
1- \left\langle \varphi_{\eta}(\xi), \varphi_{\eta}(\zeta)\right\rangle ~=~& \frac {(1- |\eta|^2)
(1- \langle \xi, \zeta\rangle)} {(1-\langle \xi, \eta\rangle) (1- \langle\eta,\zeta\rangle)}, \label{eqn:moeb0}
\intertext{and in particular,}
1- \langle \varphi_{\eta}(\xi), \eta\rangle ~=~& \frac {1-|\eta|^2} {1- \langle\xi, \eta\rangle}.\label{eqn:moeb1}
\end{align}
Finally, an easy computation shows that
\begin{equation}\label{eqn:moeb2}
1- \langle \varphi_{\eta}(\xi), \zeta \rangle ~=~ \frac {(1- \langle \xi, \varphi_{\eta}(\zeta)\rangle)
(1- \langle \eta, \zeta\rangle)} {1- \langle \xi, \eta\rangle}
\end{equation}
holds for all $\xi,\eta\in \ball$.

The best general reference here is \cite[Chapter 2]{Rud80}.

\vskip8pt
%
%

The following lemma, usually called Schur's test, is one of the most commonly used
results for proving the $L^p$-boundedness of integral
operators. See, for example, \cite[Theorem 3.6]{Zhu07}.

\begin{lem}\label{lem:schurtest}
Suppose that $(X,\mu)$ is a $\sigma$-finite measure space and
$Q(x,y)$ is a nonnegative measurable function on $X\times X$ and $T$
is the associated integral operator
\[
Tf(x)=\int_X Q(x,y) f(y) d\mu(y).
\]
Let $1<p<\infty$ and $q=p/(p-1)$. If there exist a positive constant
$C$ and a positive measurable function $g$ on $X$ such that
\[
\int_X Q(x,y) g(y)^{q} d\mu(y) \leq  C g(x)^{q}
\]
for almost every $x$ in $X$ and
\[
\int_X Q(x,y) g(x)^{p} d\mu(x) \leq  C g(y)^{p}
\]
for almost every $y$ in $X$, then $T$ is bounded on $L^p(X,\mu)$ with
$\|T\|\leq C$.
\end{lem}


\section{The proof of Key Lemma}

We begin with two lemmas.

\begin{lem}\label{lem:crucial2}
Suppose that $r,\,s>0$, $t>-1$ and $r+s-t>n+1$. Then
\begin{align*} 
\int\limits_{\ball} & \frac {(1-|\xi|^2)^t dV(\xi)}
{(1-\langle\eta, \xi\rangle)^{s} (1-\langle\zeta, \xi\rangle)^{n+1+t-s}
(1-\langle\xi, \zeta\rangle)^{n+1+t-r}}
~=~ \frac {C_1(n,r,s,t)}{4 (1-\langle\eta, \zeta\rangle)^{n+1+t-r}}
\end{align*}
holds for any $\eta\in \ball$ and $\zeta\in \sphere$.
\end{lem}

\begin{proof}

We may further assume that $r+s>2(n+1+t)$; if we prove the lemma in this special case, the general case
follows by analytic continuation.

According to \cite[Lemma 2.3]{Liu15}, the identity
\begin{align}\label{eqn:liu15}
\int\limits_{\ball} & \frac {(1-|\xi|^2)^t dV(\xi)}
{(1-\langle\eta, \xi\rangle)^{s} (1- \langle \varrho\zeta, \xi\rangle)^{n+1+t-s}
(1-\langle\xi, \varrho\zeta\rangle)^{n+1+t-r}} \notag\\
&\qquad ~=~ \frac {\pi^n \Gamma(1+t)} {\Gamma(n+1+t)} \sum_{j=0}^{\infty}
\frac {(s)_{j} (n+1+t-r)_{j}} {(n+1+t)_{j} j!} \notag \\
&\qquad \quad \times \hyperg {n+1+t-s}{n+1+t-r+j}{n+1+t+j}{\varrho^2}
(\varrho\langle\eta, \zeta\rangle)^j
\end{align}
holds for all $\varrho \in [0,1)$, $\eta\in \ball$ and $\zeta\in \sphere$.
Note that
\[
\big| \text{ the integrand in \eqref{eqn:liu15} } \big|
~\leq~ \frac {2^{r+s-2(n+1+t)} (1-|\xi|^2)^t} {|1- \langle\eta, \xi\rangle|^{s}},
\]
since $r+s>2(n+1+t)$. Letting $\varrho\to 1$, by the dominated convergence theorem and using the well-known formula
\[
\hyperg{a}{b}{c}{1} ~=~ \frac {\Gamma(c) \Gamma(c-a-b)}
{\Gamma(c-a) \Gamma(c-b)},\qquad \RePt(c-a-b)>0,
\]
we obtain
\begin{align*}
\int\limits_{\ball} & \frac {(1-|\xi|^2)^t dV(\xi)}
{(1-\langle\eta, \xi\rangle)^{s} (1-\langle\zeta, \xi\rangle)^{n+1+t-s}
(1-\langle\xi, \zeta\rangle)^{n+1+t-r}} \notag\\
&\qquad ~=~ \frac {\pi^n \Gamma(1+t)} {\Gamma(n+1+t)} \sum_{j=0}^{\infty}
\frac {(s)_{j} (n+1+t-r)_{j}} {(n+1+t)_{j} j!} \\
&\qquad \quad \times \hyperg {n+1+t-s}{n+1+t-r+j}{n+1+t+j}{1} \langle\eta, \zeta\rangle^j\\
&\qquad ~=~ \frac {\pi^{n} \Gamma(1+t)\Gamma(r+s-t-n-1)}{\Gamma(r)
\Gamma(s)} \sum_{j=0}^{\infty} \frac {(n+1+t-r)_{j}} {j!} \langle\eta, \zeta\rangle^j\\
&\qquad ~=~ \frac {C_1(n,r,s,t)}{4 (1-\langle\eta, \zeta\rangle)^{n+1+t-r}},
\end{align*}
as desired.
\end{proof}

\begin{lem}
Suppose that $r,\,s>0$, $t>-1$ and $r+s-t>n+1$. Then
\begin{align}\label{eqn:4factors1}
\int\limits_{\ball} & \frac {(1-|\omega|^2)^t dV(\omega)}
{(1-\langle\eta, \omega\rangle)^{r} (1-\langle\omega, \zeta\rangle)^{s}(1+\omega_{n})^{n+1+t-s}
(1+\overline{\omega}_{n})^{n+1+t-r}} \notag\\
& =~ \frac {C_1(n,r,s,t)}{4} (1+\eta_{n})^{s-n-1-t} (1+\overline{\zeta}_{n})^{r-n-1-t}
(1-\langle\eta, \zeta\rangle)^{n+1+t-r-s}
\end{align}
holds for all $\eta, \zeta\in \ball$.
\end{lem}

\begin{proof}

We make the change of variables $\omega=\varphi_{\eta}(\xi)$ in the integral, where
$\varphi_{\eta}$ is the M\"obius transformation of the unit ball, as defined in Section 2, as well as apply
the formulas \eqref{eqn:moeb1} and \eqref{eqn:moeb2}. After simplification, we obtain
\begin{align*}
\int\limits_{\ball} & \frac {(1-|\omega|^2)^t dV(\omega)}
{(1-\langle\eta, \omega\rangle)^{r} (1-\langle\omega, \zeta\rangle)^{s}(1+\omega_{n})^{n+1+t-s}
(1+\overline{\omega}_{n})^{n+1+t-r}} \\
& =~ (1-|\eta|^2)^{n+1+t-r} (1-\langle\eta, \zeta\rangle)^{-s}
(1+\eta_n)^{s-n-1-t} (1+\overline{\eta}_n)^{r-n-1-t}\\
& \;\; \times \int\limits_{\ball} \frac {(1-|\xi|^2)^t dV(\xi)}
{(1-\langle \xi, \varphi_{\eta}(\zeta)\rangle)^{s} (1-\langle \xi, \varphi_{\eta}(-e_n)\rangle)^{n+1+t-s}
(1-\langle \varphi_{\eta}(-e_n), \xi\rangle)^{n+1+t-r} }. \notag
\end{align*}
By Lemma \ref{lem:crucial2} and the formula \eqref{eqn:moeb0}, this equals
\begin{align*}
&(1-|\eta|^2)^{n+1+t-r} (1-\langle\eta, \zeta\rangle)^{-s}
(1+\eta_n)^{s-n-1-t} (1+\overline{\eta}_n)^{r-n-1-t}\\
&\qquad \times \frac {C_1(n,r,s,t)}{4} \ \left(1- \left\langle\varphi_{\eta}(-e_{n}),
\varphi_{\eta}(\zeta)\right\rangle\right)^{r-n-1-t}\\
=~& (1-|\eta|^2)^{n+1+t-r} (1-\langle\eta, \zeta\rangle)^{-s}
(1+\eta_n)^{s-n-1-t} (1+\overline{\eta}_n)^{r-n-1-t}\\
&\qquad \times \frac {C_1(n,r,s,t)}{4} \ \left\{
\frac {(1-|\eta|^2) (1+\overline{\zeta}_n)} {(1+\overline{\eta}) (1-\langle\eta, \zeta\rangle)}
\right\}^{r-n-1-t}
\end{align*}
which establishes the formula.
\end{proof}

Now we turn to the proof of Key Lemma.

\vskip8pt

By the change of variables $w=\Phi(\xi)$ in the integral and using \eqref{eqn:identity14phi}, we obtain
\begin{align*}
\int\limits_{\calU} &\frac {\bfrho(w)^{t}} {\bfrho(z,w)^{r} \rho(w,u)^{s}} dV(w)\\
&~=~ \int\limits_{\ball} \frac {\bfrho(\Phi(\xi))^t} {\bfrho(z,\Phi(\xi))^{r}
\bfrho(\Phi(\xi),u)^{s}} \frac {4} {|1+\xi_{n}|^{2(n+1)}} dV(\xi)\\
&~=~ 4 (1+ [\Phi^{-1}(z)]_{n})^{r} (1+ [\overline{\Phi^{-1}(u)}]_{n})^{s} \notag \\
&\qquad \times \int\limits_{\ball}  \frac {(1-|\xi|^2)^t dV(\xi)}
{(1- \langle\Phi^{-1}(z), \xi \rangle)^{r} (1- \langle\xi, \Phi^{-1}(u)\rangle)^{s}
(1+\xi_{n})^{n+1+t-s}(1+\bar{\xi}_{n})^{n+1+t-r}}.
\end{align*}
In view of \eqref{eqn:4factors1}, this equals
\begin{align*}
& C_1(n,r,s,t)\ \left(1+ [\Phi^{-1}(z)]_{n}\right)^{r} \left(1+ [\overline{\Phi^{-1}(u)}]_{n}\right)^{s}
\left(1+[\Phi^{-1}(z)]_{n}\right)^{s-n-1-t} \notag \\
& \quad \times \left(1+[\overline{\Phi^{-1}(u)}]_{n}\right)^{r-n-1-t}
\left(1- \langle\Phi^{-1}(z), \Phi^{-1}(u) \rangle \right)^{n+1+t-r-s}\\
=~ & C_1(n,r,s,t)\ \left\{ \frac {(1+ [\Phi^{-1}(z)]_{n}) (1+ [\overline{\Phi^{-1}(u)}]_{n}) }
{ 1- \langle\Phi^{-1}(z), \Phi^{-1}(u)\rangle } \right\}^{r+s-n-1-t}\\
=~ & C_1(n,r,s,t)\, \bfrho(z,u)^{n+1+t-r-s}
\end{align*}
where we used \eqref{eqn:identity14phi} to obtain the last equality. The proof is complete.


\vskip8pt

We single out a special case of Key Lemma as the following lemma, which will be used repeatedly.

\begin{lem}\label{lem:keylemma}
Let $s,t\in \bbR$. Then we have
\begin{equation}\label{eqn:keylem}
\int\limits_{\calU} \frac {\bfrho(w)^{t}} {|\bfrho(z,w)|^{s}} dV(w) ~=~
\begin{cases}
\dfrac {C_2(n,s,t)} {\bfrho(z)^{s-t-n-1}}, &
\text{ if } t>-1 \text{ and } s-t>n+1\\[12pt]
+\infty, &  otherwise
\end{cases}
\end{equation}
for all $z\in \calU$, where
\[
C_2(n,s,t):=\frac {4 \pi^{n} \Gamma(1+t) \Gamma(s-t-n-1)} {\Gamma^2\left(s/2\right)}.
\]
\end{lem}

\begin{proof}
It remains to show that the integral is finite if and only if $t > -1$ and $s-t > n+1$.

Before proceeding, we recall the definition of the Heisenberg group and some basic facts which can be found in
\cite[Chapter XII]{Ste93}.

We denote by $\bbH^{n-1}$ the Heisenberg group, that is, the set
\[
\bbC^{n-1} \times \bbR = \{ [\zeta,t]: \zeta\in \bbC^{n-1}, t\in \bbR\}
\]
endowed with the group operation
\[
[\zeta,t]\cdot [\eta,s]=[\zeta+\eta, t+s+2 \mathrm{Im}\langle\zeta,\eta\rangle)].
\]
To each element $h=[\zeta,t]$ of $\bbH^{n-1}$, we associate the following (holomorphic) affine self-mapping of
$\calU$:
\begin{equation}\label{eqn:groupaction}
h:\; (z^{\prime},z_{n}) \longmapsto (z^{\prime}+\zeta, z_{n}+t+ 2i \langle z^{\prime}, \zeta\rangle + i|\zeta|^2).
\end{equation}
It is easy to check that
\begin{equation}\label{eqn:h-inv}
\bfrho(h(z),h(w))=\bfrho(z,w)
\end{equation}
for any $z,w\in \calU$ and any $h\in \bbH^{n-1}$.

For fixed $z\in \calU$, we put $h=[-z^{\prime}, -\mathrm{Re} z_{n}]\in \bbH^{n-1}$.
It is easy to check that $h(z) = \bfrho(z)\bfi$, where $\bfi=(0^{\prime},i)$, and
\[
\bfrho(h(z),w) = \frac {i}{2}(\overline{w}_n-\bfrho(z)i)
\]
for all $w\in \calU$. Using \eqref{eqn:h-inv} and making the change of variables $w\mapsto h(w)$ in the integral, we see that
\begin{align*}
\int\limits_{\calU} \frac {\bfrho(w)^{t}} {|\bfrho(z,w)|^{s}} dV(w)
~=~& \int\limits_{\calU} \frac {\bfrho(w)^{t}} {|\bfrho(h(z),w)|^{s}} dV(w) \\
~=~& 2^s \int\limits_{\calU} \frac {(\ImPt w_n -|w^{\prime}|^2)^{t}} {|w_n + \bfrho(z) i|^{s}} dV(w).
\end{align*}
By Fubini's theorem, this equals
\begin{align*}
& 2^s \, \int\limits_{\ImPt w_n >0} \frac {1} {|w_n + \bfrho(z) i|^{s}}
\Bigg\{\int\limits_{|w^{\prime}|< (\ImPt w_n)^{1/2}} (\ImPt w_n -|w^{\prime}|^2)^{t} dm_{2n-2}(w^{\prime}) \Bigg\} dm_2(w_n)\\
&\quad \quad ~=~
2^s \, \Bigg\{\int\limits_{\ImPt w_n >0} \frac {(\ImPt w_n)^{n-1+t}} {|w_n + \bfrho(z) i|^{s}} dm_2(w_n)\Bigg\}
\Bigg\{\int\limits_{|w^{\prime}|<1} (1 -|w^{\prime}|^2)^{t} dm_{2n-2}(w^{\prime}) \Bigg\},
\end{align*}
which is finite if and only if $t> -1$ and $s-(n-1+t)> 2$.
\end{proof}

\section{The proof of Theorem \ref{thm:main}}

\subsubsection*{(ii) $\Rightarrow$ (i):} Obvious.

\subsubsection*{(i) $\Rightarrow$ (iii):}
Suppose that $T$ is bounded on $L^p(\calU, dV_{\alpha})$.

\vskip8pt

\noindent \textit{Case 1: $p=\infty$.}
Note that the constant function $\mathbf{1}$ cannot serve as a test function at this moment,
since $T \mathbf{1}(z)\equiv 0$.
Instead, we consider the function
\[
f_z(w):=\frac {\bfrho(z,w)^{c}} {|\bfrho(z,w)|^{c}}, \qquad w\in \calU.
\]
Each $f_z$ is a unit vector in $L^{\infty}(\calU)$ and
\[
(T f_z)(z)= \bfrho(z)^{a} \int\limits_{\calU} \frac { \bfrho (w)^{b}} { |\bfrho(z,w)|^{c}} dV(w)
\]
for every $z\in \calU$. Since $|(T f_z)(z)|\leq \|T\|_{\infty\to \infty}$ for all $z\in \calU$,
where  $\|T\|_{\infty\to \infty}$ denotes the operator norm of $T$ acting on $L^{\infty}(\calU)$,
by Lemma \ref{lem:keylemma}, we have
\[
\begin{cases}
b > -1,\\
c>n+1+b, \\
c-n-1-b = a,
\end{cases}
\]
which is clearly nothing but \eqref{eqn:condns4infty}.

\vskip8pt

\noindent \textit{Case 2: $p=1$.}
Note that the boundedness of  $T$ on $L^1(\calU, dV_{\alpha})$  implies the boundedness of $T^{\ast}$
on  $L^{\infty}(\calU)$, where $T^{\ast}$ is the adjoint of $T$. It is easy to see that
\begin{equation}\label{eqn:adjoint}
T^{\ast} f(z) = \bfrho(z)^{b-\alpha} \int\limits_{\calU} \frac { \bfrho (w)^{a+\alpha}} { \bfrho(z,w)^{c}} f(w) dV(w).
\end{equation}
So we can apply the previous case to $T^{\ast}$ to obtain
\[
\begin{cases}
a+\alpha>-1,\\
c>n+1+(a+\alpha), \\
c-n-1-(a+\alpha) = b-\alpha,
\end{cases}
\]
which implies
\[
\begin{cases}
-a < \alpha + 1 < b+1,\\
c=n+1+a+b.
\end{cases}
\]

\vskip8pt

\noindent \textit{Case 3: $1<p<\infty$.}

We first show that $c>0$.
In order that $Tf$ be always well-defined for
 $f\in L^p(\calU,dV_{\alpha})$,  it is necessary and sufficient that
 \[
 \int\limits_{\calU} \frac {\bfrho(w)^{bq+\alpha}} {|\bfrho(z,w)|^{cq}} dV(w) ~<~ +\infty
 \]
 for all $z\in \calU$, where $q:p/(p-1)$ is the conjugate exponent of $p$.
 Again by Lemma \ref{lem:keylemma},  this happens if and only if
 \[
 \begin{cases}
bq+\alpha > -1,\\
cq - bq -\alpha > n+1.
\end{cases}
 \]
Summing up the two inequalities, we get  $c > n/q >0$.

For $\beta>0$, we put
 \[
 f_{\beta}(z) ~:=~ \frac {\bfrho(z)^t} {\bfrho(z,\beta \bfi)^{s}}, \qquad z\in \calU,
 \]
 where $s,t$ are real parameters satisfying the conditions
 \begin{align}
s ~>~& 0, \tag{C.1}  \\
t ~>~& \max \left\{-\frac {1+\alpha}{p},\ -1-b\right\}, \tag{C.2} \\
s-t ~>~& \max \left\{ \frac {n+1+\alpha}{p},\ n+1+b-c\right\}. \tag{C.3}
 \end{align}
By Lemma \ref{lem:keylemma}, Conditions (C.1)--(C.3) guarantee that $f_{\beta}\in L^p(\calU, dV_{\alpha})$ and
 \begin{equation}\label{eqn:normoff_st}
 \|f_{\beta}\|_{p,\alpha}^{p} ~=~ C_3 (n,\alpha,p,s,t)\, \beta^{n+1+\alpha-p(s-t)},
 \end{equation}
where
\[
C_3 (n,\alpha,p,s,t) := \frac {4\pi^n \Gamma(pt+1+\alpha) \Gamma (p(s-t)-n-1-\alpha)}{\Gamma^2(ps/2)}.
\]
Also, in view of Conditions (C.1)--(C.3) and that $c>0$, we can apply Key Lemma to obtain
 \begin{align*}\label{eqn:Tf_st}
 (Tf_{\beta})(z) ~=~& \bfrho(z)^{a} \int\limits_{\calU} \frac {\bfrho(w)^{b+t}} {\bfrho(z,w)^{c} \bfrho(w,\beta\bfi)^{s}} dV(w) \notag\\
 =~& C_4(n,b,c,s,t) \frac {\bfrho(z)^{a}}{\bfrho(z, \beta\bfi)^{c-b-n-1+s-t}},
 \end{align*}
where
\[
C_4(n,b,c,s,t) ~:=~ \frac {4\pi^n \Gamma(b+t+1) \Gamma (c-b-n-1+s-t)}{\Gamma(c)\Gamma(s)}.
\]
Since $T f_{\beta}\in  L^p(\calU, dV_{\alpha})$, again by Lemma \ref{lem:keylemma},
it is necessary that
\begin{gather}
pa+\alpha>-1,  \label{eqn:paracondn1}\\
p(c-a-b-n-1) + p(s-t) -n-1-\alpha>0. \notag
\end{gather}
Moreover, we have
\begin{equation}\label{eqn:normofTf}
 \|Tf_{\beta}\|_{p,\alpha}^{p} ~=~ C_5(n,\alpha,p,b,c,s,t)\, \beta^{n+1+\alpha-p(s-t)+p(n+1+a+b-c)},
\end{equation}
where $C_5(n,\alpha,p,b,c,s,t)$ equals
\begin{align*}
& C_4(n,b,c,s,t)^p \\
& \quad \times \frac {4\pi^n \Gamma(1+pa+\alpha) \Gamma (p(c-a-b-n-1+s-t)-n-1-\alpha)}
 {\Gamma^2(p(c-b-n-1+s-t)/2)}.
\end{align*}
Since $T$ is bounded on $L^p(\calU, dV_{\alpha})$, there is a positive constant
$C$, independent of $\beta$, such that $\|T f_{\beta}\|_{p,\alpha} \leq C \|f_{\beta}\|_{p,\alpha}$
for all $\beta\in (0,\infty)$. Taking \eqref{eqn:normoff_st} and \eqref{eqn:normofTf} into account,
we can find another positive constant $C^{\prime}$,
independent of $\beta$, such that
\[
\beta^{n+1+\alpha-p(s-t)+p(n+1+a+b-c)} ~\leq~ C^{\prime}\, \beta^{n+1+\alpha-p(s-t)}
\]
for all $\beta\in (0,\infty)$. But this is true only when $c=n+1+a+b$.

Having proved that $c=n+1+a+b$ and $-pa<\alpha+1$, we proceed to show that $\alpha+1<p(b+1)$.
Note that the boundedness of $T$ on $L^p(\calU, dV_{\alpha})$ is
equivalent to the boundedness of $T^{\ast}$ on $L^q(\calU, dV_{\alpha})$,
where $T^{\ast}$ is the adjoint of $T$, as is given by \eqref{eqn:adjoint}.
Applying \eqref{eqn:paracondn1} to $T^{\ast}$, we conclude that
\[
\alpha+1 > -q(b-\alpha),
\]
which is exactly the same as
\[
\alpha+1 < p(b+1).
\]

%

\subsubsection*{(iii) $\Rightarrow$ (ii):}

The cases $p=1$ and $p=\infty$ are direct consequences
of Lemma \ref{lem:keylemma}.

In the case $1 < p < \infty$, the proof appeals to Schur's test. Let
\[
Q(z,w)= \frac {\bfrho(z)^{a} \bfrho (w)^{b-\alpha}} {|\bfrho(z,w)|^{n+1+a+b}}.
\]
and $g(z)=\bfrho(z)^{-(1+\alpha)/(pq)}$, where $q=p/(p-1)$.
Again, it follows from Lemma \ref{lem:keylemma} that
\begin{align*}
\int\limits_{\calU} Q(z,w) & g(w)^{q} \bfrho(w)^{\alpha} dV(w) \\
~=~& \bfrho(z)^{a}
\int\limits_{\calU} \frac {\bfrho(w)^{b-(1+\alpha)/p}} {|\bfrho(z,w)|^{n+1+a+b}} dV(w)\\
=~& \bfrho(z)^{a}\, \frac {4\pi^n \Gamma(1+b-(1+\alpha)/p) \Gamma(a+(1+\alpha)/p)} {\Gamma^2((n+1+a+b)/2)}\,
\bfrho(z)^{-a-(1+\alpha)/p}\\
~=~& \frac {4\pi^n \Gamma(1+b-(1+\alpha)/p) \Gamma(a+(1+\alpha)/p)} {\Gamma^2((n+1+a+b)/2)}\, g(z)^{q}
\end{align*}
holds for every $z\in \calU$. Similarly,
\begin{align*}
\int\limits_{\calU} Q(z,w) g(z)^p \bfrho(z)^{\alpha}dV(z)
~=~ \frac {4\pi^n \Gamma(1+b-(1+\alpha)/p) \Gamma(a+(1+\alpha)/p)} {\Gamma^2((n+1+a+b)/2)}\, g(w)^{p}
\end{align*}
holds for every $w\in \calU$. Hence, by Lemma \ref{lem:schurtest}, $S$ is bounded
on $L^p(\calU, dV_{\alpha})$ with
\[
\|S\| ~\leq~ \frac {4\pi^n \Gamma(a+(1+\alpha)/p) \Gamma(1+b-(1+\alpha)/p)} {\Gamma^2((n+1+a+b)/2)}
\]
The proof is complete.

\section{Applications}


We present two examples to illustrate the use of our main result.

In order to state the first example we need to introduce more notation.
It is known that the Bergman kernel function $K_{\Omega}$ induces a Riemannian
metric on a domain $\Omega$ in $\bbC^n$. The infinitesimal Bergman metric is defined by
\[
g_{i,j}^{\Omega}(z) = \frac {1}{n+1} \frac {\partial^2 \log K_{\Omega}(z,z)}{\partial z_i \partial \bar{z}_j},
\quad i,j=1,2,\ldots,n,
\]
and the complex matrix
\[
B(z)=\big(g_{i,j}^{\Omega}(z)\big)_{1\leq i,j \leq n}
\]
is called the Bergman matrix of $\Omega$.
For a $C^1$ curve $\gamma:[0,1]\to \Omega$, the Bergman length of $\gamma$ is defined by
\[
\ell(\gamma):= \int\limits_0^1 \langle B(\gamma(t)) \gamma^{\prime}(t), \gamma^{\prime}(t) \rangle dt.
\]
If $z,w\in \Omega$, then their Bergman distance is
\[
\delta_{\Omega}(z,w):=\inf \{\ell(\gamma): \gamma(0)=z, \gamma(1)=w \},
\]
where the infimum is taken over all $C^1$ curves from $z$ to $w$.
If $\Omega_1$, $\Omega_2$ are two domains in $\bbC^n$ and $\psi$ is a
biholomorphic mapping of $\Omega_1$ onto $\Omega_2$, then $\delta_{\Omega_1}(z,w)=
\delta_{\Omega_2}(\psi(z),\psi(w))$ for all $z,w\in \Omega_1$.
Hence,
\[
\delta_{\calU}(z,w) = \delta_{\ball}(\Phi^{-1}(z),\Phi^{-1}(w))
= \tanh^{-1} \left(\left|\varphi_{\Phi^{-1}(z)}(\Phi^{-1}(w)\right|\right).
\]
Furthermore, a computation shows that
\begin{equation}\label{eqn:hyperdist}
\delta_{\calU}(z,w)=\tanh^{-1} \sqrt{1-\frac {\bfrho(z)\bfrho(w)}{|\bfrho(z,w)|^2}}.
\end{equation}
%

Let $a, b$ and $c$ be real numbers. We consider the operator
\[
S_{a,b}^{c} f(z) := \bfrho(z)^{a} \int\limits_{\calU} \frac { \bfrho (w)^{b} \delta_{\calU}(z,w)^c} {|\bfrho(z,w)|^{n+1+a+b}} f(w) dV(w).
\]
It is a modification of the integral operator $S_{a,b,c}$ in Theorem \ref{thm:main},
with an extra unbounded factor $\delta(z,w)^c$ in the integrand.

\begin{thm}\label{thm:appl1}
Suppose $\alpha\in \bbR$ and $1\leq p <\infty$. If $-pa<\alpha+1<p(b+1)$ and $c\geq 0$ then
the operator $S_{a,b}^{c}$ is bounded on $L^p(\calU, dV_{\alpha})$.
\end{thm}

\begin{proof}
Pick $\epsilon >0$ so small that $-p(a-c\epsilon)<\alpha+1<p(b+1-c\epsilon)$.
Since $\log x < x^{\epsilon}$ holds for any $x>0$ and any $\epsilon>0$, it follows from \eqref{eqn:hyperdist} that
\begin{align*}
\delta_{\calU}(z,w) ~\lesssim~& \log \frac {4|\bfrho(z,w)|^{2}} {\bfrho(z) \bfrho(w)} ~\lesssim~ 1+ \frac {|\bfrho(z,w)|^{2\epsilon}} {\bfrho(z)^{\epsilon} \bfrho(w)^{\epsilon}}.
\end{align*}
It follows that
\begin{align*}
|S_{a,b}^{c} (f)(z)| ~\lesssim~& |S_{a,b,n+1+a+b} (|f|)(z)|\\
& \quad +  \bfrho(z)^{a-c\epsilon} \int\limits_{\calU} \frac { \bfrho (w)^{b-c\epsilon}}
{|\bfrho(z,w)|^{n+1+a+b-2c\epsilon}} |f(w)| dV(w) \\
~=~& |S_{a,b,n+1+a+b} (|f|)(z)| + |S_{\tilde{a},\tilde{b},n+1+\tilde{a}+\tilde{b}} (|f|)(z)|,
\end{align*}
where $\tilde{a}=a-c\epsilon$ and $\tilde{b}=b-c\epsilon$. The desired result then follows from Theorem \ref{thm:main}.
\end{proof}


We denote by $A_{\alpha}^p(\calU)$ the Bergman space, that is, the closed subspace of $L^p(\calU, \bfrho^{\alpha})$ consisting of
holomorphic functions on $\calU$. As usual,  we write $\partial_n:=\partial/(\partial z_n)$.
The following result plays an important role in the study of the Besov spaces over the Siegel upper half-space.

\begin{thm}
Suppose $1< p < \infty$, $\alpha > 1/p-1$ and $N\in \bbN$.
Then $\partial_n^N$ is a bounded
linear operator from $A_{\alpha}^p(\calU)$ into $A_{\alpha+pN}^p (\calU)$.
\end{thm}


\begin{proof}
According to \cite[Theorem 2.1]{DK93}, if $f\in A_{\alpha}^p(\calU)$ with $p$ and $\alpha$ satisfying
the assumption of the theorem, then
\[
f(z)= c_{\alpha} \int\limits_{\calU} f(w) \frac {\bfrho(w)^{\alpha}} {\bfrho(z,w)^{n+1+\alpha}} dV(w),
\]
where
\[
c_{\alpha} ~:=~ \frac {\Gamma(n+1+\alpha)} {4\pi^n \Gamma(1+\alpha)}.
\]
It follows that
\begin{align*}
\big|\bfrho(z)^N \partial_n^N f(z)\big| ~\lesssim~& \bfrho(z)^N \int\limits_{\calU} |f(w)| \frac {\bfrho(w)^{\alpha}}
{|\bfrho(z,w)|^{n+1+\alpha+N}} dV(w) \\
=~& S_{N,\alpha,n+1+\alpha+N}(|f|)(z).
\end{align*}
By Theorem \ref{thm:main}, this implies
\[
\|\partial_n^N f\|_{p,\alpha+pN} ~=~ \|\bfrho^N \partial_n^N f\|_{p,\alpha} ~\lesssim~ \|f\|_{p,\alpha},
\]
as asserted.
\end{proof}

\subsection*{Acknowledgement}
We are grateful to an anonymous referee for several valuable suggestions and especially
for pointing out a gap in the proof of Theorem 1 in the original version of this paper.
We also wish to thank Professor H. Turgay Kaptanoglu for constructive comments and for
bringing the paper by Ruhan Zhao to our attention.

\end{document}